\documentclass[11pt]{amsart}
\usepackage[hmarginratio=1:1]{geometry}              
\usepackage[parfill]{parskip}    
\usepackage{graphicx}
\usepackage{amssymb}
\usepackage{epstopdf}
\usepackage{amsthm} 
\usepackage{marvosym}
\usepackage{enumerate}
\usepackage[all]{xy}
\usepackage[linktocpage]{hyperref}
\usepackage{float}
\usepackage{bbm}
\usepackage{mathtools}

\newtheorem{theorem}{Theorem}
\numberwithin{theorem}{section}

\newtheorem{maintheorem}{Theorem}

\newtheorem{lemma}[theorem]{Lemma}

\newtheorem{claim}[theorem]{Claim}

\newtheorem{proposition}[theorem]{Proposition}

\newtheorem{corollary}[theorem]{Corollary}
\newtheorem{question}[theorem]{Question}

\theoremstyle{remark}
\newtheorem{rmk}[theorem]{Remark}

\theoremstyle{definition}
\newtheorem{definition}[theorem]{Definition}

\newcommand{\C}{\mathbb{C}}

\newcommand{\Z}{\mathbb{Z}}
\newcommand{\Q}{\mathbb{Q}}

\newcommand{\D}{\mathbb{D}}
\newcommand{\PP}{\mathbb{P}}

\newcommand{\Aut}{\text{Aut}}

\newcommand{\Mod}{\text{Mod}}

\newcommand{\id}{\text{id}}

\newcommand{\Sp}{\text{Sp}}
\newcommand{\U}{\text{U}}
\newcommand{\im}{\text{im}}
\newcommand{\overl}{\overline}
\newcommand{\Alb}{\text{Alb}}

\title{On Kodaira fibrations with invariant cohomology}
\author{Corey Bregman}

\begin{document}
\begin{abstract} A Kodaira fibration is a compact, complex surface admitting a holomorphic submersion onto a complex curve, such that the fibers have nonconstant moduli. We consider Kodaira fibrations $X$ with nontrivial invariant $\Q$-cohomology in degree 1, proving that if the dimension of the holomorphic invariants is 1 or 2, then $X$ admits a branch covering over a product of curves inducing an isomorphism on rational cohomology in degree 1. We also study the class of Kodaira fibrations possessing a holomorphic section, and demonstrate that having a section imposes no restriction on possible monodromies.

\end{abstract}
\maketitle
\section{Introduction}
A \textit{Kodaira fibration} is a holomorphic submersion $f:X\rightarrow C$, from a complex surface $X$ onto a complex curve $C$, whose fibers are diffeomorphic to the orientable surface $\Sigma_h$ of genus $h$, but have nonconstant moduli. As a moduli problem, the data of a Kodaira fibration is the same as a nonconstant holomorphic map from a curve $C$ into the moduli space $\mathcal{M}_h$ of curves of genus $h$.   Thus, the problem of constructing Kodaira fibrations is the same as that of finding complete curves in moduli space.  More generally, as $\mathcal{M}_h$ is the classifying space for all surface bundles with genus $h$ fibers; that is, such bundles are classified up to homotopy equivalence by continuous maps of surfaces into $\mathcal{M}_h$, up to homotopy.  In this sense, Kodaira fibrations provide explicit examples of such extensions, equipped with a natural geometry.  

The topological properties of a Kodaira fibrations are determined entirely by the monodromy homomorphism $\widehat{\rho}:\pi_1(C)\rightarrow \Mod(\Sigma_h)$, where $\Mod(\Sigma_h)$ is the genus $h$ mapping class group. In what follows, we will be concerned chiefly with the corresponding representation on degree 1 cohomology, denoted $\rho:\pi_1(C)\rightarrow \Aut(H^1(\Sigma_h;\Z))\cong\Sp_{2h}(\Z)$.  We study Kodaira fibrations in which $\rho$ has nontrivial invariants. Let the holomorphic part of the invariant cohomology in degree 1 be denoted $H^{1,0}(\Sigma_h)^\rho$. We prove:

\begin{maintheorem}\label{main1}Let $f:X\rightarrow C$ be a Kodaira fibration with $\dim_\C H^{1,0}(\Sigma_h)^\rho=d\leq 2$.  Then there exists a curve $D$ of genus $d$ and a ramified covering $F:X\rightarrow D\times C$ such that $F^*:H^1(D\times C;\Q)\rightarrow H^1(X;\Q)$ is an isomorphism.
\end{maintheorem}

Kodaira fibrations are named after Kodaira, who constructed the first examples in \cite{Ko67}.  Most notably, Kodaira showed these examples have non-vanishing signature, demonstrating that the signature need not be multiplicative for general fiber bundles.  The construction was discovered independently by Atiyah \cite{At69}, and later refined by Hirzebruch \cite{Hir69}.   Various constructions of Kodaira fibrations were given subsequently by Kas \cite{Kas68}, Riera \cite{R77}, Gonzalez Diez--Harvey \cite{GH91}, Zaal \cite{Z95}, Bryan--Donagi \cite{BD02}, Catanese-Rollenske \cite{CaRo09}, and most recently by Flapan \cite{Fl17} and Lee--L\"onne--Rollenske \cite{LLR17}.  

The constructions listed above are fairly explicit, often formed by taking a suitable ramified covering of a product of curves.  One key feature that they have in common, implied by the branched covering, is the existence of nontrivial invariants in degree 1 $\Q$-cohomology.  Alternate constructions of Kodaira fibrations can be obtained by finding curves in moduli space in a much less explicit fashion; namely, by intersecting the Satake compactification of a finite orbifold covering of $\mathcal{M}_h$ with multiple hyperplanes in general position (see, for example \cite{Ar17}) in order to obtain a curve.  We call fibrations constructed in this way \textit{general complete intersection Kodaira fibrations}. This technique produces bundles with monodromy a finite index subgroup of the mapping class group.  In particular, the bundle will typically have no invariants in degree 1 $\Q$-cohomology.  

Our motivation for Theorem \ref{main1} therefore came from a desire to distinguish these two distinct constructions of Kodaira fibrations on the level of cohomology.  Curiously, although Flapan's examples are formed by the latter method, \textit{i.e.} taking a generic complete intersection, they have 1-dimensional holomorphic invariants, hence also arise from a branched covering construction by Theorem \ref{main1}. 

The first step in the proof of Theorem \ref{main1} works whenever there are nontrivial invariants.  In general, however, the target will not be a product, but instead another Kodaira fibration $p:Y\rightarrow C$ which is in some sense \textit{reduced} (Definition \ref{Reduced}) in that it immerses in its Albanese:

\begin{maintheorem}\label{main2} Let $f:X\rightarrow C$ be a Kodaira fibration with invariants $H^{1,0}(\Sigma_h)^\rho\neq 0$.  There exists a Kodaira fibration $p:Y\rightarrow C$ and a (ramified) covering $F:X\rightarrow Y$ satisfying \[\xymatrix{X\ar[rr]^F\ar[dr]_f&&Y\ar[dl]^p\\ &C&}\] and such that $Y$ immerses in its Albanese $\Alb(Y)$.  
\end{maintheorem}

In addition to invariant cohomology, the Atiyah--Kodaira examples  come equipped with holomorphic (multi)sections. This property holds more generally for all double \'etale Kodaira fibrations as in \cite{CaRo09}. On the group-theoretic level, their fundamental groups split as a semidirect product of the fundamental group of the base and the fiber.  For general smooth real surface-bundles over real surfaces, the group theoretic splitting of the fundamental group is equivalent to the bundle possessing a smooth section.  However, a holomorphic section seems more restrictive.  In \S4 we analyze Kodaira fibrations with holomorphic sections and show that for questions of monodromy there is no crucial difference between having a holomorphic section or not:
\begin{maintheorem}\label{main3}Let $f:X\rightarrow C$ be a Kodaira fibration with fiber genus $h\geq 2$.  There exists a branched covering $\widetilde{C}\rightarrow C$ and a pullback square \[\xymatrix{
\widetilde{X}\ar[r]\ar[d]_{\widetilde{f}}&X\ar[d]^f\\
\widetilde{C}\ar[r]&C
}\]such that $\widetilde{f}:\widetilde{X}\rightarrow \widetilde{C}$ has a section and the images of $\pi_1(C)$ and $\pi_1(\widetilde{C})$ in $\Mod(\Sigma_h)$ agree.  
\end{maintheorem} 
%

\textbf{Outline:} In \S2 we review background on Kodaira fibrations, fibered K\"ahler manifolds and discuss the Albanese map of a variety. In \S3, we prove Theorems \ref{main1} and \ref{main2}, and deduce some group theoretical consequences.  In \S4, we study Kodaira fibrations with a holomorphic section, and establish Theorem \ref{main3}. Finally, in \S5, we make some concluding remarks and propose questions for further study. 

\textbf{Acknowledgements:} I would like to thank several people on ``both sides" for many helpful comments and conversations.  In particular, I am indebted to Jason Starr for offering help on many points of algebraic geometry, and to Laure Flapan for describing her examples to me.  I am also grateful to Rob Lazarsfeld, Andy Putman, Nick Salter and Bena Tshishiku for many helpful comments and suggestions.

\section{Background}
In this section, we establish some notation and review some general results on K\"ahler groups and Kodaira fibrations which will frame our discussion in the sequel.
\begin{definition} A group $\Gamma$ is \textit{K\"ahler} if $\Gamma\cong \pi_1(X)$ for some compact K\"ahler manifold $X$.  
\end{definition}
It follows from compactness that any such $\Gamma$ is necessarily finitely presentable. The K\"ahler condition often places strong restrictions on $\Gamma$.  For an excellent introduction to K\"ahler groups see \cite{ABCKT96}.  

Among the simplest examples of K\"ahler manifolds are Riemann surfaces, or smooth, complex projective curves.  Accordingly, by \textit{surface group} we mean the fundamental group of a closed, orientable manifold of real dimension 2. Real surfaces are classified up to homeomorphism by the topological genus $g$.  Throughout, we denote the topological surface of genus $g$ by $\Sigma_g$ and its fundamental group by $\Pi_g=\pi_1(\Sigma_g)$. Given a Riemann surface $C$, we use the notation $g(C)$ for its topological genus.  

Surface groups hold a special place among all K\"ahler groups because of the following 
\begin{theorem}[Siu, Beauville \cite{ABCKT96}] \label{SiuBeauville}$X$ admits a surjective holomorphic map to some Riemann surface $C$ of genus $g'\geq g\geq 2$ with connected fibers if and only if there is a surjective homomorphism $\psi:\pi_1(X)\cong\Gamma\rightarrow \Pi_g$.
\end{theorem}

Whether or not a compact K\"ahler manifold admits a surjective holomorphic map to a high genus curve is determined entirely by the fundamental group. The crucial idea in the above theorem dates back to a classical result of Castelnuovo--de Franchis: a surjective holomorphic map to a curve pulls back an \textit{isotropic} subspace of holomorphic 1-forms.  Conversely, an isotropic subspace of holomorphic 1-forms provides a holomorphic map to some projective space $\mathbb{P}^n$ with 1-dimensional image.  

\subsection{Surface-by-surface groups and the monodromy representation}

\begin{definition}A \textit{surface-by-surface group} $\Gamma$ is any group which fits into a short exact sequence
\begin{equation}1\rightarrow\Pi_h\rightarrow\Gamma\xrightarrow{\phi} \Pi_g\rightarrow 1 
\label{eq:MainSequence}\tag{$*$}\end{equation} where $g,h\geq 2$. 
\end{definition}
Topologically, extensions of the form $(*)$ are represented by surface-bundles over surfaces. In particular, $\Gamma$ will have a real 4-dimensional Eilenberg--Maclane space.  In what follows, we consider the case when $\Gamma$ as in $(*)$ is K\"ahler. Applying Theorem \ref{SiuBeauville} to $\phi$, we conclude that there exists a Riemann surface $C'$ of genus $g'\geq g$, together with a surjection $f:X\rightarrow C'$. In fact, because $\Pi_h$ is finitely presentable, we have a stronger result

\begin{proposition}[Kotschick \cite{Kot99}] There exists a surjective holomorphic map $f:X\rightarrow C$ to a curve $C$ of genus $g$ with connected fibers inducing $\phi$ on $\pi_1$. 
\end{proposition}
When $X$ is 2-dimensional, \textit{i.e.} a complex surface, the map $f$ gives $X$ the structure of Riemann surface bundle over a Riemann surface, possibly with finite many singular fibers. If there are no singular fibers, we recall the following characterization due independently to Hillman, Kapovich and Kotschick:

\begin{theorem}[Hillman \cite{Hi00}, Kapovich \cite{K98}, Kotschick \cite{Kot99}]\label{HKK} Let $f:X\rightarrow C$ be as above.  $X$ is aspherical if and only if $f$ is a holomorphic submersion.
\end{theorem}
Since the target is complex 1-dimensional, here a submersion means that the holomorphic differential $Df$ does not vanish at any point of $X$. In this case, by Ehresmann's fibration theorem, $X$ is the total space of $C^\infty$ locally trivial fiber bundle over $C$, where the fibers are closed orientable surfaces of genus $h$. Since $X$ has a complex structure, each fiber will therefore be a Riemann surface of genus $h$. 

\begin{definition} A \textit{Kodaira fibration} is a complex surface $X$ together with a holomorphic submersion with connected fibers $f:X\rightarrow C$, which is not locally trivial as a holomorphic bundle.  
\end{definition}

It is known that if $f:X\rightarrow C$ is a Kodaira fibration then the genus of $C$ is at least $2$, and the genus of the fibers is at least 3 \cite{Kas68}. By the classification of surfaces (see, for example \cite{Beau83} ), any such $X$ will be of general type, hence K\"ahler and projective. 

The bundle structure of $X$ is determined up to diffeomorphism \cite{Kot99} by the \textit{monodromy representation} $\widehat{\rho}:\Pi_g\rightarrow \Mod(\Sigma_h)$, where $\Mod(\Sigma_h)$ is the mapping class group of $\Sigma_h$.  The induced action on cohomology is trivial on $H^0(\Sigma_h;\Z)$ and $H^2(\Sigma_h;\Z)$ and preserves the natural cup product pairing on $H^1(\Sigma_h;\Z)$.  Since the latter is skew-symmetric and we obtain a monodromy representation $\rho:\Pi_g\rightarrow \Sp_{2h}(\Z)$.  

\subsection{Invariant cycles theorem}For any point $t\in C$, we denote by $X_t$ the fiber $f^{-1}(t)$, with inclusion map $\iota:X_t\hookrightarrow X$. Since $X$ is K\"ahler, $H^*(X;\Q)$ admits a Hodge structure and $f^*$, $\iota^*$ are morphisms of Hodge structures.  The Leray--Serre spectral sequence associated to $f$ degenerates at the $E^2$-page (See Voisin \cite{Vo03}, Morita \cite{Mor87}), and we have a five-term short exact sequence of rational cohomology:
\[0\rightarrow H^1(C;\Q)\xrightarrow{f^*} H^1(X;\Q)\xrightarrow{\iota^*} H^1(X_t;\Q)^{\rho}\rightarrow 0.\]
This is a special (trivial) case of the invariant cycles theorem. 

Since $X$ is projective, we can choose a polarization $Q$ on $H^*(X;\Q)$ and split the above exact sequence after tensoring with $\Q$:\[H^1(X;\Q)\cong V_\Q\oplus W_\Q,\] where $W_\Q=\im(f^*)\cong H^1(C;\Q)$ and $V_\Q$ is a weight-1 $\Q$-Hodge substructure of $H^1(X;\Q)$ which maps isomorphically onto $H^1(X_t;\Q)^{\rho}$ under $\iota^*$. Upon tensoring with $\C$, we can decompose $V$ and $W$ further into holomorphic and antiholomorphic components:\[V_\C=V^{1,0}\oplus V^{0,1}\]\[W_\C=W^{1,0}\oplus W^{0,1}.\]

Given any $\theta\neq\in V^{1,0}$, the restriction $\iota^*(\theta)$ is a holomorphic 1-form on $X_t$ and $\iota^*(\theta\wedge\overl{\theta})$ generates the fiber class $H^2(X_t;\Q)$.  We record here the following important consequence of this fact:
\begin{proposition}\label{LerayHirsch} For any $\theta\in V^{1,0}$ and any $\eta\in W^{1,0}$, $\theta\wedge\eta\neq 0$.  
\end{proposition}
\begin{proof} Consider the (2,2)-form $\nu=(\theta\wedge \eta)\wedge (\overl{\theta}\wedge\overl{\eta})=-(\theta\wedge\overl{\theta})\wedge(\eta\wedge\overl{\eta})$.  At almost every point of $X$, $\eta\wedge\overl{\eta}\neq0$ and vanishes on the vertical tangent bundle.  At almost every point of $X$, $\theta\wedge\overl{\theta}\neq0$ and does not vanish on the vertical tangent bundle.  Since the tangent bundle splits at each point into vertical and horizontal components, we see that $\nu$ does not vanish identically, and since it is harmonic, it is nontrivial in cohomology.  Thus, $\theta\wedge \eta$ is nontrivial as well.  
\end{proof} 

\begin{definition}A Kodaira fibration $f:X\rightarrow C$ is called \textit{branched} if there exists another curve $D\neq C$ of genus at least one and a finite-degree branched covering $F:X\rightarrow C\times D$ such that $f=p_C\circ F$, and $g=p_D\circ F:X\rightarrow D$ have connected fibers. $F$ is \textit{maximal} if $F^*$ is an isomorphism on $H^1$ with $\Q$-coefficients.  
\end{definition}
\begin{rmk}\label{Isotropic}
If $f:X\rightarrow C$ is branched, then $g^*H^{1,0}(D)$ is a maximal isotropic subspace of $H^{1,0}(X)$, such that $g^*H^{1,0}(D)\cap f^*H^{1,0}(C)=0$.  Conversely, if $L\subseteq H^{1,0}(X)$ is an isotropic subspace of dimension $\geq 2$ and $L\cap f^*H^{1,0}(C)=0$, then $L$ defines a map to a curve and  guarantees $X$ is branched. 
\end{rmk}

\subsection{The Albanese map}\label{Albanese} Integration of 1-forms over cycles in $H_1(X;\Z)$ induces a map \[j:H_1(X;\Z)\rightarrow H^{1,0}(X)^\vee\] as follows. Any homology class $[c]\in H_1(X;\Z)/\{\text{torsion}\}$ induces a homomorphism $H^{1,0}(X)\rightarrow \C$ via:\[j:c\mapsto \left(\gamma\mapsto\int_c\gamma\right)\]

From Hodge theory, the image of $j$ is a lattice in $H^{1,0}(X)^\vee=(V^{1,0})^\vee\oplus (W^{1,0})^\vee$.  We will analyze the map $\alpha_X:X\rightarrow \Alb(X)$ defined as follows.  Fix a basepoint $x_0\in X$, given any $x\in X$ and any holomorphic 1-form $\gamma\in H^{1,0}(X)$, define \[\alpha_X(x)=\int_{x_0}^x\gamma \] where the integral is taken along any path from $x_0$ to $x$.  This is well-defined mod $\im(j)$ since the integral over a different path will differ by an element of the lattice. Recall that $\Alb(X)$ is universal with respect to holomorphic maps from $X$ to any abelian variety. In other words, if $A$ is an abelian variety and $h:X\rightarrow A$ is any holomorphic map, then there is a unique holomorphic map $\alpha_X(h):\Alb(X)\rightarrow A$ such that 
\[\xymatrix{& \Alb(X)\ar[d]^{\alpha_X(h)}\\
X\ar[ur]^{\alpha_X}\ar[r]^h&A}\]
commutes.

\section{Kodaira fibrations with invariant cohomology}
Let $f:X\rightarrow C$ be a Kodaira fibration with monodromy representation $\rho$.  Choose $t\in C$ and assume that $H^1(X_t;\Q)^\rho\neq 0$. By the invariant cycles theorem above, $H^1(X_t;\Q)^\rho=\im(\iota^*)$ has a Hodge structure and hence there is a decomposition\[H^1(X_t;\C)^\rho=H^{1,0}(X_t)^\rho\oplus H^{0,1}(X_t)^\rho.\]
We call $H^{1,0}(X_t)^\rho$ the \textit{holomorphic invariants} of $f:X\rightarrow C$.

\begin{definition}\label{Reduced}A Kodaira fibration $f:X\rightarrow C$ is called \textit{reduced} if the Albanese map $\alpha_X:X\rightarrow \Alb(X)$ is an immersion.  
\end{definition}

Our first result shows that every Kodaira fibration admits a (ramified) covering over a reduced Kodaira fibration.
\begin{theorem}\label{ReducedBranched}Suppose $f:X\rightarrow C$ is a Kodaira fibration with $H^{1,0}(X_t)^\rho\neq 0$.  Then there exists a reduced Kodaira fibration $p:Y\rightarrow C$ and a (ramified) covering $F:X\rightarrow Y$ such that $f=p\circ F$ and $F^*$ is an isomorphism on $H^1$ with $\Q$-coefficients.
\end{theorem}

If the dimension of  $H^{1,0}(X_t)^\rho$ is sufficiently small, the reduced Kodaira fibration above will actually be a product.  
\begin{theorem}\label{BranchedCriterion}Suppose $f:X\rightarrow C$ is a Kodaira fibration with  $\dim_\C H^{1,0}(X_t)^\rho=1,2$.
Then there exists a curve $D$ and a branched cover $F:X\rightarrow D\times C$ such that 
\begin{enumerate}
\item $f=p_C\circ F$ where $p_C:D\times C\rightarrow C$ is the projection onto $C$.
\item $F^*$ induces an isomorphism $H^1(D\times C;\Q)\xrightarrow{\sim} H^1(X;\Q).$
\end{enumerate}
\end{theorem}
\begin{rmk}In general, $F$ will not be a double Kodaira fibration, let alone double \'etale (see \cite{CaRo09} for a definition).  For instance, if $\dim_\C H^{1,0}(X_t)^\rho=1$, then $D$ is an elliptic curve.  However, for a Kodaira fibration, the base curve must have genus at least 2.  Examples where $\dim_\C H^{1,0}(X_t)^\rho=1$ have been constructed by Flapan \cite{Fl17}.  
\end{rmk}
%
We also obtain the following group-theoretic corollary. As far as the author is aware, the maximal, branched Kodaira fibrations are the only known K\"ahler groups with infinite abelianization and which are residually finite, but not residually torsion-free nilpotent. Residual finiteness follows from the fact that $\Pi_h$ and hence $\Aut(\Pi_h)$ are both residually finite.  
\begin{corollary}\label{Nilpotent}With $X$ as in Theorem \ref{BranchedCriterion}, $\pi_1(X)$ is not residually nilpotent.
\end{corollary}
\begin{proof} For any group $G$, we denote by $G^{(n)}$ the $n$th term in the lower central series, and set $G^{(\infty)}=\bigcap_{i=1}^{\infty}G^{(n)}$. If $G$ is residually nilpotent then $G^{(\infty)}\cong 1$.  

Let $\Gamma=\pi_1(X)$ and $\Pi=\pi_1(D\times C)$. Observe that $F^*$ induces an isomorphism on $H^1(-;\Q)$ and an injection on $H^2(-;\Q)$. Therefore, by Stallings' theorem on lower central series with $\Q$-coefficients (\cite{St65}, Theorem 7.3), $F_*:\pi_1(X)\rightarrow \pi_1(D\times C)$ induces injections on all nilpotent quotients and an injection $\Gamma/\Gamma^{(\infty)} \hookrightarrow \Pi/\Pi^{(\infty)}$.  Now $\Pi$ is residually nilpotent since it is a product of surface groups, which are residually nilpotent.  If $\Gamma$ were residually nilpotent, then $F_*:\Gamma\hookrightarrow \Pi$ would be an injection.  Since $F$ has finite degree, this implies $\Gamma$ is a finite index subgroup of $\Pi$.  As $X$ and $D\times C$ are aspherical, we conclude that $X$ is homotopy equivalent to a finite cover of $D\times C$, a contradiction.  
\end{proof}

The rest of the section will be devoted to proving Theorems \ref{ReducedBranched} and \ref{BranchedCriterion}.  The strategy will be to construct a branched covering from $X$ onto another surface $Y$.  We then show that $Y$ is smooth and also admits the structure of a Kodaira fibration $p:Y\rightarrow C$.  However, $X$ and $Y$ will have isogenous Albanese varieties and the map $Y\hookrightarrow \Alb(Y)$ will be an immersion.

To construct $Y$, we first consider the canonical map from $X$ to its Albanese torus $\Alb(X)$. In sections \ref{MappingAlbanese} and \ref{MappingProduct} we work in greater generality and do not assume that $\dim_\C H^{1,0}(X_t)^\rho\leq2$, only that $\dim_\C H^{1,0}(X_t)^\rho\neq 0$. 
\subsection{Mapping to the Albanese}\label{MappingAlbanese}
Choose a basepoint $t_0\in C$ and a preimage $x_0\in f^{-1}(t_0)$ to define Albanese maps $\alpha_C$, $\alpha_X$ as in \S \ref{Albanese}. By the universal property, we obtain a commutative diagram \[\xymatrix@=40pt{
X\ar[d]_f\ar[r]^{\alpha_X}&\Alb(X)\ar[d]^{\alpha(f)}\\
C\ar[r]^{\alpha_C}&\Alb(C)
}\]
We remark that $\alpha(f)$ is a surjective submersion since $f:X\rightarrow C$ is, and by our choice of basepoint, $\alpha(f)$ is a homomorphism with respect to the group structures on $\Alb(X)$ and $\Alb(C)$.  

Let $A$ denote the kernel of $\alpha(f)$. By Poincar\'e's complete reducibility theorem (\cite{Mum08}, \S19 Theorem 1), there exists an abelian subvariety $Z\subseteq \Alb(X)$ such that $A+Z=\Alb(X)$ and $A\cap Z$ is finite.  Therefore there is a finite covering $q':A\times Z \rightarrow \Alb(X)$. The restriction of $\alpha(f)$ to $Z$ induces a finite covering $Z\rightarrow \Alb(C)$; hence $Z$ is isogenous to $\Alb(C)$, and there is an integer $n_Z$ such that the multiplication map $1\times n_Z:A\times \Alb(C)\rightarrow A\times Z$ is an isogeny.  Let $q=q'\circ (1\times n_Z)$ be the composite isogeny $q:A\times \Alb(C)\rightarrow \Alb(X)$.  From this description, it is clear that:
\begin{enumerate}
\item $q^*$ induces isomorphisms on $H^*$ with $\Q$-coefficients. 
\item The composition \[A\times\{0\}\hookrightarrow A\times\Alb(C)\xrightarrow{q}\Alb(X)\xrightarrow{\alpha(f)} \Alb(C)\] induces the zero map on $H^*$ with $\Q$-coefficients.
\item The composition \[\{0\}\times \Alb(C)\hookrightarrow A\times\Alb(C)\xrightarrow{q}\Alb(X)\xrightarrow{\alpha(f)}\Alb(C)\] induces an isomorphism on $H^*$ with $\Q$-coefficients.
\end{enumerate}

Let $p_A:A\times \Alb(C)\rightarrow A$ be the projection onto the first factor. We can build a map $\alpha:X\rightarrow A$ as follows.  Given $x\in \Alb(x)$, $q^{-1}(x)=\{x_1,\ldots, x_d\}$, where $d$ is the degree of $q$.  Define $r:\Alb(X)\rightarrow A\times\Alb(C)$ by \[r(x)=\sum_{i=1}^dx_i\]
Then we set $\alpha(x)=p_A\circ r\circ \alpha_X(x)$.   
\begin{proposition}\label{virtualSplit}$H^1(X;\Q)\cong \im(\alpha^*)\oplus \im( f^*)$ as weight-1 $\Q$-Hodge structures. 
\end{proposition}
\begin{proof} Observe that $r$ is an isogeny since clearly $q\circ r$ is just multiplication by $d$.  In particular, $\alpha_X^*\circ r^*$ is an isomorphism of $\Q$-Hodge structures. On the other hand, $H^1(A\times \Alb(C);\Q)$ can be canonically identified with $H^1(A;\Q)\oplus H^1(\Alb(C);\Q)$ by projection. Now we have the following commutative diagram where all cohomology groups are with $\Q$-coefficients:
\[\xymatrix@=40pt{
H^1(A)\ar[r]^{p_A^*}\ar[d]_{\alpha^*}& H^1(A\times\Alb(C))\ar@/^/[d]^{r^*}&H^1(\Alb(C))\ar[l]_{p_{\Alb(C)}^*}\\
H^1(X)&H^1(\Alb(X))\ar[u]^{q^*}\ar[l]_{\alpha_X^*}&\\
H^1(C)\ar[u]^{f^*}&H^1(\Alb(C))\ar[l]_{\alpha_C^*}\ar[u]^{\alpha(f)^*}&
}\]
By $(2)$ and $(3)$ above, we see that $\im(p_A^*)\cap\im(q^*\circ\alpha(f)^*)=0$ and $\im(p_A^*)+\im(q^*\circ\alpha(f)^*)=H^1(A\times\Alb(C);\Q)$. From the commutativity of the lower square and the fact that $\alpha_X^*\circ r^*$ is an isomorphism, the proposition follows. 
\end{proof}

From the proposition, instead of choosing an arbitrary splitting of $\iota^*$, we can use $V=\im(\alpha^*)$ and $W=\im(f^*)$. We remark that since $X$ is a surface, the image of $\alpha$ is either 1-dimensional or 2-dimensional.  In fact, it is not difficult to show that $V$ comes from a curve if and only if the image of $X$ in $A$ is 1-dimensional, in which case the image of $X$ will be a smooth curve in $A$.

\subsection{Mapping into a product}\label{MappingProduct}
Consider the product map $F'=\alpha\times f:X\rightarrow A\times C$, and let $Y'=\im(F')$.  By a slight abuse of notation, we will use $F'$ either to refer to the map to $A\times C$ or to $Y'\subset A\times C$ Projecting onto $C$, we see that also $Y'$ fibers over $C$ since we have a commutative diagram
\[\xymatrix{
X\ar[dr]_f\ar[r]^{F'}&Y'\ar[d]^{p'}\\
& C
}\]

By Chow's theorem (see for example, \cite{GH94}), $Y'$ is an algebraic subvariety of $\Alb(X)$.  From Proposition \ref{virtualSplit}, the composition \[X\xrightarrow{F'}A\times C \rightarrow A\times\Alb(C)\] induces an isomorphism on $H^1$ with $\Q$-coefficients. As $F'$ factors through $Y'$, the same is true of $F':X\rightarrow Y'$.  First, we claim that $Y'$ is generically smooth. We have
\begin{lemma}\label{mostlySmooth} The map $F':X\rightarrow Y'$ is finite. Away from a finite set of points: 
\begin{enumerate}
\item the ramification locus is transverse to the fibers of $f$, and
\item $Y'$ is locally the image of a smooth surface immersed in $A\times C$.
\end{enumerate}
\end{lemma}
\begin{proof} The fibers of $F'$ are either 0 or 1-dimensional. Suppose for some $y\in Y'$, $F'^{-1}(y)$ contains a curve $B$. Let $t=p'(y)$. Since $B\subseteq F'^{-1}(y)\subseteq f^{-1}(t)=X_{t}$, we see that $B=X_t$ since $X_t$ is irreducible.  But then $X\rightarrow Y'\hookrightarrow A\times C$ cannot induce an isomorphism on $H^1$.  This contradiction proves the first statement in the lemma.  

By the previous paragraph, we see that $F'$ is a finite degree branched covering and the fibers of $p':Y'\rightarrow C$ are connected, possibly singular algebraic curves. We next analyze what happens when $F'$ is singular.  

Let $R\subset X$ be the ramification divisor.  Then $R$ is an algebraic subvariety of $X$ by Chow's theorem again and hence generically smooth. Consider the projection $p_C:A\times C\rightarrow C$ and the composition $X\xrightarrow{F'} A\times C\xrightarrow{p_C} C$.  Then since $p_C\circ F'=f$, we have that $\ker(DF')\subseteq \ker(Df)$. Since the latter is 1-dimensional everywhere, the former is at most 1-dimensional.  For $x_0\in R$, we must have that $\ker(DF')=\ker(Df)$ is exactly 1-dimensional.  
\begin{claim} $F'$ is not singular along a fiber of $f$.  
\end{claim}
\begin{proof} By definition, $Df$ vanishes along the vertical tangent bundle.  Since $F'=\alpha\times f$, to prove the claim it suffices to show that $D\alpha$ does not identically vanish along the tangent bundle of any fiber of $f$. This is equivalent to showing that for any fixed fiber $X_t$, the induced map on the cotangent space $D^*(\alpha\circ\iota):T^*A\rightarrow T^*X_t$ is does not vanish. But for any $\theta\in \im(\alpha^*)$, we know that $\theta$ restricts to $X_t$ as a nontrivial holomorphic 1-form. Therefore $\iota^*(\theta)$ has only finitely many zeroes and does not vanish on $T^*X_t.$  
\end{proof}

We conclude that $R$ does not contain any fiber.  In fact, for each fixed $t\in C$, the above shows that $R\cap X_t$ is the set of common zeroes of the restriction of $\im(\alpha^*)$ to $X_t$. Next we show that locally $Y'$ is the immersed image of a complex surface with only has isolated singularities.  The singularities occur only at the image of nonsmooth points of $R$ or where $R$ is tangent to a fiber $X_t$.
\begin{claim} If $R$ is smooth at a point $x_0\in X_{t_0}$ and not tangent to $X_{t_0}$, then there exists a neighborhood $V$ of $x_0$ on $R$ which is transverse to the fibers of $f$, and such that the $F'(V)$ is a smooth surface embedded in $A\times C$.  
\end{claim}    
\begin{proof}If $x_0$ is a smooth point and $R$ is not tangent to $X_{t_0}$ at $x_0$, by the preceding discussion this means exactly that $T_{x_0}R$ exists and projects isomorphically onto $T_{t_0}C$. By the holomorphic inverse function theorem, there exists a neighborhood $U\ni t_0$ and a neighborhood $V\ni x_0$ and a section $\sigma:U\rightarrow V\cap R$. In particular, $V\cap R$ is transverse to the fibers of $f$. We may choose $V$ to be fiber-preserving $C^\infty$-diffeomorphic to $\D^2\times \D^2$ where $\D^2\times \{0\}$ corresponds to $V\cap R$. Setting $V'=V\setminus V\cap R$, we see that $F'|_{V'}$ is a covering map, preserving the fibers, and moreover, $F'|_{V\cap R}$ is a fiber-preserving diffeomorphism, as $p_2\circ F'\circ \sigma=\id_U$. Since $\pi_1(V')\cong \Z$, we conclude that $F'|_V$ is a fiberwise branched cover, branched along $V\cap R$.  Hence, the image $F'(V)$ is smooth by Mumford's topological criterion \cite{Mum61}.
\end{proof}
The proof of the claim obviously implies that $Y'$ is the immersed image of surface which is smooth except possibly where $R$ is not smooth or where $R$ is tangent to the fibers of $f$. Both of these possibilities only occur at finitely many isolated points.  
\end{proof}

In the case where $\dim_\C H^{1,0}(X_t)^\rho=1$, then $A=E$ an elliptic curve, and $E\times C$ is a smooth surface. Hence, $F'$ is onto and 
\begin{corollary}\label{ellipticCase} If $\dim_\C H^{1,0}(X_t)^\rho=1$, then $F':X\rightarrow E\times C$ is a finite-degree branched covering inducing an isomorphism on $H^1$ with $\Q$-coefficients.
\end{corollary}

This proves Theorem \ref{BranchedCriterion} for the case when $\dim H^{1,0}(X_t)^\rho=1$.  In general, $\dim_\C(A)>1$, and we do not know that the image is smooth.  However, we will show that it is the image of a smooth Kodaira fibration immersed in $A\times C$.  To do this, we show that the normalization of $Y'$ is smooth. Since $Y'$ is projective, it has a unique normalization $Y$ together with a finite map $\tau:Y\rightarrow Y'$. As $X$ is smooth, it is normal and hence the map $F'$ factors as $\tau\circ F$ for some map $F:X\rightarrow Y$.  It is well-known that $Y$ only has isolated singularities. 

\begin{lemma}\label{SmoothImage}$Y$ is smooth.  
\end{lemma}
\begin{proof}
Write $p=p'\circ\tau:Y\rightarrow C$. A result of Brieskorn (\cite{Br67}, Satz 2.8) states that if $g:W\rightarrow Z$ is a finite, surjective morphism from a smooth surface $W$ to a normal surface $Z$, then $Z$ only has quotient singularities. Thus, $Y$ has isolated quotient singularities, which are classified by finite subgroups of $\U(2)$. 

Let $y_0\in Y$ be a singular point.  There exists an open subset $V\ni y_0$ which can be identified with $\Delta/G$ for some finite group $G\leq \U(2)$, where $\Delta\subset \C^2$ is an open ball centered on the origin.  Let $\pi$ be the projection $\pi:\Delta\rightarrow V$ and denote by $\widehat{p}$ the composition $p\circ \pi$.  Choose $x_0\in X$ mapping onto $y_0$ and let $t_0=f(x_0)=p(y_0)$ be their common image point in $C$. Now let $U\ni x_0$ be a fibered neighborhood of $x_0$ satisfying: 
\begin{enumerate}
\item $x_0$ is the only preimage of $y_0$ contained in $U$,
\item $U$ projects to $W'\subseteq W$ under $f$,
\item $F(U)\subseteq V$.  
\end{enumerate}
Topologically, $U$ is a trivial $\D^2$-bundle over $\D^2$.  

As $y_0$ is an isolated singularity, $G$ acts freely on $\Delta^*=\Delta\setminus \{0\}$, which makes $\pi:\Delta^*\rightarrow V^*=V\setminus\{y_0\}$ into the universal covering. On $\Delta^*$, the composition $\widehat{p}=\pi\circ p$ is a smooth submersion which is $G$-equivariant.  By Hartog's theorem, $\widehat{p}$ extends to a \emph{holomorphic} function $\widehat{p}:\Delta\rightarrow W$. Since $U^*=U\setminus\{x_0\}$ is simply connected, $F:U^*\rightarrow V^*$ lifts to a holomorphic map $\widehat{F}:U^*\rightarrow \Delta^*$. Applying Hartog's theorem again, $\widehat{F}$ extends to a holomorphic map $\widehat{F}:U\rightarrow \Delta$.  The two holomorphic functions $\widehat{p}\circ \widehat{F}$ and $f$ agree on an open subset of $U$, namely $U^*$, hence by the uniqueness of analytic continuation, they agree on all of $U$.   
Then we have a commutative diagram:\[
\xymatrix{
 &\Delta \ar[d]^\pi\ar@/^1.5pc/[dd]^{\widehat{p}} \\
U \ar[r]^{F} \ar[ur]^{\widehat{F}} \ar[dr]^f& V\ar[d]^{p}\\
& W
}\]

By chain rule, we see that $\widehat{p}$ is a submersion at $0$, since $f=\widehat{F}\circ \widehat{p}$ is a submersion $x_0$ and $\widehat{F}$ is holomorphic on a neighborhood of $0\in \Delta$.  
The key observation is that $\widehat{p}$ is $G$-equivariant. In particular, the fiber $\Delta_{0}$ of $\widehat{p}$ above $t_0$ is a smooth curve, invariant under the action of $G$.  Let $D\subseteq \Delta_0$ be an open disk neighborhood of $0$.  Set $D_0=\bigcap_{g\in G}g.D$.  Then $D_0$ is a neighborhood of $0$ which is $G$-invariant. Since $G$ fixes $0$ and acts freely on $\Delta^*$, we see that $G$ also acts freely on $D_0^*$.  We conclude that $G$ is finite cyclic.  It is easy to see that $G$ is conjugate to a diagonal subgroup generated by  \[\text{diag}(\omega,\omega^k)=\left(\begin{array}{cc}
\omega &0\\
0 & \omega^k
\end{array}\right)\]
where $\omega=e^{\frac{2\pi i}{n}}$ and $(n,k)=1$. We now compute $D\widehat{p}(0)$ in a different way.  Let $(x,y)$ be local coordinates on $\Delta$ so that the action of $G$ has the above form, and let $w$ be a local coordinate near $t_0\in W$, centered on $t_0$.  Since $\widehat{p}$ is analytic at $0$, on a neighborhood  of $0$ we can write \[w=\widehat{p}(x,y)= ax +by + \text{higher order terms},\] where at least one of $a,b\neq 0$. Since $\widehat{p}$ is $G$-equivariant,  we obtain \[\widehat{p}(x,y)=\widehat{p}(\omega x, \omega^ky)= (a\omega )x +(b\omega^k)y+ \text{higher order terms}\] Setting these two power series equal to each other, we deduce that $\omega=1$ and $G$ is trivial.  Thus, $y_0$ is a smooth point, as desired.
\end{proof}

From the lemma, we have that $p$ is a submersion since $f=p\circ F$ is a submersion. In particular, $Y$ also has the structure of a Kodaira fibration. We record this in 
\begin{proposition}\label{SmoothNormal} The normalization $Y$ is smooth and $p:Y\rightarrow C$ is a submersion. 
\end{proposition}

We can now prove Theorem \ref{ReducedBranched}:
\begin{proof}Given $f:X\rightarrow C$ we have constructed $F:X\rightarrow Y$, where $p:Y\rightarrow C$ is also a Kodaira fibration and $f=p\circ F$ by Proposition \ref{SmoothNormal}. As $F'$ is an isomorphism on $H^1$ with $\Q$-coefficients, so is $F$.  Since $Y$ is the normalization of $Y'$, $\tau:Y\rightarrow A\times C$ is an immersion and hence $F'$ is an immersion if and only if $F$ is.  Thus, if $F'$ is not an immersion, then $F:X\rightarrow Y$ is a nontrivial branched covering and the genus of the fibers of $p$ are strictly smaller than those of $f$.  To see $Y$ is reduced, consider the composition $Y\rightarrow A\times C\rightarrow A\times \Alb(C)$. The second map is an embedding since $C\hookrightarrow \Alb(C)$ is an embedding. By the universal property of $\Alb(Y)$, there is an induced diagram \[ \xymatrix{&\Alb(Y)\ar[d]\\
Y\ar[ur]^{\alpha_Y}\ar[r]&A\times \Alb(C)}\]
The vertical map induces an isomorphism on $H^1$ with $\Q$-coefficients, hence must be an isogeny.  On the other hand, since $Y\rightarrow A\times\Alb(C)$ is an immersion, so is $\alpha_Y$.
\end{proof}

\subsection{The case when $\dim_\C H^{1,0}(X_t)^\rho=2$}
We proved the dimension 1 case of Theorem \ref{BranchedCriterion} in Corollary \ref{ellipticCase}. The remaining case is when $\dim_\C H^{1,0}(X_t)^\rho=2$:

\begin{proof}By Theorem \ref{ReducedBranched}, we may assume that $X$ is reduced, and that $F':X\rightarrow A\times C$ is an immersion. Since $\dim_\C(A)=2$, $X$ has codimension 1 in $A\times C$.  Let $\nu$ be the normal bundle to $X$ in $A\times C$. The Chern character of the tangent bundle to $A$ is trivial, and the Chern character of the tangent bundle to $C$ is $1-K_C$, where $K_C$ is the canonical class of $C$. Pulling back to $X$ we obtain \[c(\nu\oplus TX)=(1+c_1(\nu))\cdot(1+c_1(X)+c_2(X))=1-f^*K_C\]
From this we get two equations:
\begin{equation}c_1(\nu)=-f^*K_C-c_1(X),
\end{equation}
\begin{equation}c_1(X)c_1(\nu)+c_2(X)=0,
\end{equation}
and substituting the former into the latter we obtain:
\begin{equation}c_2(X)-c_1(X)\cdot f^*K_C-c_1^2(X)=0.
\end{equation}
Now $f^*K_C$ is represented by the divisor $-\chi(C)X_t$ where $X_t$ is a fiber of $f$.  Since $X_t\cdot X_t=0$, by adjunction we have 
\begin{equation} -c_1(X)\cdot X_t=K_X\cdot X_t=-\chi(X_t)
\end{equation}
Equations (3) and (4) together with the fact that $c_2(X)=\chi(C)\chi(X_t)$ yield
\begin{equation}c_1^2(X)=2c_2(X)
\end{equation}
By the Hirzebruch signature theorem, Equation (5) implies that the signature of $X$ vanishes and hence that $X$ must be a product.  Since $F':X\rightarrow A\times C$ is an isomorphism on $H^1$ with $\Q$-coefficients, this means that $X\cong D\times C$ where $D$ is a genus 2 curve.  
\end{proof}

\section{Kodaira fibrations with a section}
The original Atiyah--Kodaira examples also possessed a holomorphic multisection. This property holds more generally for double \'etale Kodaira fibrations as in \cite{CaRo09}. 
\begin{definition} A Kodaira fibration $f:X\rightarrow C$ has a \textit{multisection} if there exists a finite (unramified) covering $p:\widetilde{C}\rightarrow C$ and a lift $s:\widetilde{C}\rightarrow X$ such that the following diagram commutes: 
\[\xymatrix{&X\ar[d]^f\\
\widetilde{C}\ar[r]^p\ar[ur]^s&C
}\]
\end{definition}
By the universal property of fiber products, $s:\widetilde{C}\rightarrow X$ is a multisection if and only if the pullback bundle $p^*(X)\rightarrow \widetilde{C}$ has a section.  Therefore, after passing to a cover of the base, we may assume the bundle has a section. Since $C$ has genus $\geq 2$, the existence of a smooth section is a group-theoretic condition: $f:X\rightarrow C$ has a section if and only if there is a homomorphism $\sigma:\pi_1(C)\rightarrow \pi_1(X)$ such that $\sigma\circ f_*=\id$. Although this seems like a strong condition, in fact having a multisection is not that restrictive from the point of view of monodromy.

\begin{proposition}\label{BranchedSection}For every Kodaira fibration $f:X\rightarrow C$ there exists a Kodaira fibration $\widetilde{f}:\widetilde{X}\rightarrow \widetilde{C}$ such that:\begin{enumerate}[(i)]
\item There are branched coverings $p:\widetilde{C}\rightarrow C$ and $P:\widetilde{X}\rightarrow X$ such that \[\xymatrix{
\widetilde{X}\ar[r]^P\ar[d]_{\widetilde{f}}&X\ar[d]^f\\
\widetilde{C}\ar[r]^p&C
}\]is a pullback square. In particular, the fibers of $f$ and $\widetilde{f}$ have the same genus $h$
\item The monodromy homomorphisms $\rho:\pi_1(C)\rightarrow \Mod(\Sigma_h)$ and $\widetilde{\rho}:\pi_1(\widetilde{C})\rightarrow \Mod(\Sigma_h)$ have the same image.  
\item $\widetilde{f}:\widetilde{X}\rightarrow \widetilde{C}$ has a holomorphic section.
\end{enumerate}Moreover, if we pass to a finite unramified cover of $X$, we can choose $p$ and $P$ to be regular coverings. 
\end{proposition}
\begin{rmk}It is not difficult to find a surjection from $\pi(\Sigma_N)\rightarrow \pi_1(X)$ for some $N\gg g$, by say factoring through a free group.  However, this will never be homotopic to a holomorphic map. In particular, the euler class of the pullback bundle will vanish since it factors through a free group.
\end{rmk}
\begin{proof} Embed $X$ into some projective space $\PP^N$ for $N$ large.  By Bertini's theorem \cite{GH94}, for generic $H\cong \PP^2\subset\PP^N$, $H\cap X$ will be a smooth, connected curve $\widetilde{C}$. The projection $p=f|_{\widetilde{C}}:\widetilde{C}\rightarrow C$ is a finite degree branched covering.  Let $\widetilde{X}:=p^*(X)$ be the pullback.  All the fibers of $\widetilde{f}:\widetilde{X}\rightarrow \widetilde{C}$ have genus $h$, and we have maps $\id:\widetilde{C}\rightarrow \widetilde{C}$ and $j:\widetilde{C}\rightarrow X$ such that $p\circ\id=f\circ j$.  By the universal property of fiber products there is a holomorphic map $s:\widetilde{C}\rightarrow \widetilde{X}$
\[
\xymatrix{\widetilde{C}\ar@{.>}[dr]^{s}\ar@/^/[drr]^{j}\ar@/_/[ddr]_{\id}& & \\
& \widetilde{X}\ar[r]^{P}\ar[d]^{\widetilde{f}}&X\ar[d]^{f}\\
& \widetilde{C}\ar[r]^{p}& C
}\]
Thus, $\widetilde{f}\circ s=\id$ so $s$ is a holomorphic section. This proves (i) and (iii). 

Let $\rho:\pi_1(C)\rightarrow \Mod(\Sigma_h)$ and $\widetilde{\rho}:\pi_1(\widetilde{C}\rightarrow \Mod(\Sigma_h)$ be the monodromy homomorphisms.  By the Lefschetz hyperplane theorem, $j:\widetilde{C}=H\cap X\rightarrow X$ induces a surjection $j_*:\pi_1(\widetilde{C})\rightarrow \pi_1(X)$. In particular $\Pi_h\leq \im j_*$ and $\im\widetilde{\rho}\cong \pi_1(X)/\Pi_h=\im(\rho)$.  This proves (ii). For the last statement, given any branched covering of curves $C'\rightarrow C$ it is easy to pass to a further branched covering which is regular and whose image in $\pi_1$ is a finite index subgroup of $\pi_1(C)$.
\end{proof}

Proposition \ref{BranchedSection} clearly establishes Theorem \ref{main3}. Note that $P^*$ carries the holomorphic invariants of $f$ isomorphically onto the holomorphic invariants of $\widetilde{f}$.  As noted in Remark \ref{Isotropic}, $f:X\rightarrow C$ is branched if and only if $H^{1,0}(X_t)^\rho$ is 1-dimensional or contains an isotropic subspace of dimension $\geq 2$. The following corollary is therefore immediate:
\begin{corollary} $\widetilde{X}$ is branched if and only if $X$ is branched.
\end{corollary}

\section{Concluding remarks} It is natural to wonder whether Theorem \ref{BranchedCriterion} can be extended to cases of higher dimensional invariants.  As far as the author is aware, all examples of Kodaira fibrations in the literature either come branch over a product of curves, or have trivial invariants in first cohomology. 

\begin{question}Does there exist a Kodaira fibration with nontrivial invariants that does not admit a branched covering over a product of curves?
\end{question}

If $\dim_\C H^{1,0}(X_t)^\rho\geq2$, it is not hard to see that a Kodaira fibration $f:X\rightarrow C$ admits a branched covering over a product of curves if and only if there exist $\beta_1,\beta_2\in H^{1,0}(X_t)^\rho$ such that $\beta_1\wedge \beta_2=0$.  That is, there is an isotropic subspace of the invariants of dimension at least 2.  Any two such isotropic subspaces are disjoint if they are maximal.  Hence repeating the construction of \S3, any $X$ maps onto a product of curves and a torus.   

In terms of cohomology, the map $\alpha:X\rightarrow A$ induces an isomorphism on $H^1$ and a surjection on $H^2$.  The extent to which $\alpha^*$ fails to be an injection in degree 2 determines the lower central series of $\pi_1(X)$.  Based on this observation, we are able to extend Corollary \ref{Nilpotent} to the case $\dim_\C H^{1,0}(X)^\rho=3$.

\begin{proposition} If $f:X\rightarrow C$ is a Kodaira fibration with $\dim_\C H^{1,0}(X)^\rho=3$, then $\pi_1(X)$ is not residually nilpotent.
\end{proposition}
\begin{proof} Based on Theorem \ref{BranchedCriterion}, there is a map $F:X\rightarrow C\times Y$ where $Y$ is either
\begin{enumerate}[(i)]
\item A curve of genus 3.
\item A product of a curve of genus 2 and an elliptic curve. 
\item A 3-dimensional torus.
\end{enumerate}
In the first two cases, the map induces an injection on $H^2$, and thus, by the same argument in Corollary \ref{Nilpotent}, $\pi_1(X)$ is not residually nilpotent.  In the last case, any two linearly independent elements of $H^{1,0}(X)^\rho$ have nontrivial cup product.  Since every element of $H^{2,0}(Y)$ is decomposable, the map $H^2(Y;\Q)\rightarrow H^2(X;\Q)$ is an injection.  Hence in this case, $\pi_1(X)$ has the same rational lower central series as $\pi_1(C)\times \pi_1(Y)\cong \Pi_g\times \Z^6$, and is likewise not residually nilpotent. 
\end{proof}

This is some indication that $\pi_1(X)$ may not be residually nilpotent for any Kodaira fibration $X$. However, the lower central series of $\pi_1(X)$ may not always be the same as that of a product of curves and a torus.  If a vector space $V$ has dimension at least 4, there are elements of $\wedge^2(V)$ which are not decomposable, so the above argument does not work.  

\begin{question}Is $\pi_1(X)$ ever residually nilpotent, for $X$ a Kodaira fibration? More generally, what are the possible 1-rational homotopy types of $X$? \end{question}

\bibliography{KahlerbibSurf}
\bibliographystyle{plain}

\end{document}